\documentclass[11pt]{amsart}
\theoremstyle{plain}
\usepackage{curves}
\usepackage{amsmath, amsfonts, amsthm, amscd, amssymb}
\usepackage[all]{xy}
\title{A remark on the topology of $(n,n)$ Springer varieties}

\author{Stephan M. Wehrli}
\address{Institut de Math\'ematiques de Jussieu; Universit\'e Paris 7; 175 rue du Chevaleret; bureau 7B3; 75013 Paris, France}
\email{wehrli@math.jussieu.fr}

\theoremstyle{plain}
\newtheorem{theorem}{Theorem}[section]
\newtheorem{lemma}[theorem]{Lemma}
\newtheorem{proposition}[theorem]{Proposition}
\theoremstyle{definition}

\newcommand{\C}{\ensuremath{\mathbb{C}}}

\begin{document}
\bibliographystyle{halpha}

\begin{abstract} We prove a conjecture of Khovanov~\cite{khovanov-2004-2}
which identifies the topological space underlying the Springer variety of complete flags in $\mathbb{C}^{2n}$ stabilized by a fixed nilpotent operator with two Jordan blocks of size $n$.
\end{abstract}
\maketitle


\section{Introduction}

Let $E_n$ be a complex vector space of dimension $2n$ and $z_n\colon E_n\rightarrow E_n$ a nilpotent linear endomorphism with two nilpotent Jordan blocks, each of them of size $n$. A {\em complete flag} in $E_n$ is an ascending sequence of linear subspaces $0\varsubsetneq L_1\varsubsetneq L_2\varsubsetneq\ldots\varsubsetneq L_{2n}=E_n$. The $(n,n)$ {\em Springer variety} is the set
$$
\mathfrak{B}_{n,n}:=\{\mbox{complete flags in $E_n$ stabilized by $z_n$}\},
$$
where a complete flag is said to be {\em stabilized} by $z_n$ if each of the subspaces $L_j$ is stable under $z_n$, i.e. if $z_nL_j\subset L_j$ for all $j\in\{1,\ldots,2n\}$.

It is known that $\mathfrak{B}_{n,n}$ is a complex projective variety of (complex) dimension $n$, and that the irreducible components of $\mathfrak{B}_{n,n}$ are topologically trivial (but algebraically non-trivial) iterated $\mathbb{P}^1$-bundles over a point (where $\mathbb{P}^1$ is the complex projective line, i.e., topologically, $\mathbb{P}^1\cong S^2$). Moreover, a result of Fung~\cite{fung-2002} (going back to earlier work of Spaltenstein~\cite{spaltenstein-1976} and Vargas~\cite{vargas-1979}), describes the irreducible components of $\mathfrak{B}_{n,n}$ explicitly in terms of crossingless matchings of $2n$ points:
\begin{proposition}[Fung]\label{prop:fung}
The irreducible components of $\mathfrak{B}_{n,n}$ are
parametrized by crossingless matchings of $2n$ points. Furthermore, the irreducible component $K_a$ associated to $a\in B^n$ can be described explicitly, as follows:
$$
K_a=\{(L_1,\ldots,L_{2n})\in\mathfrak{B}_{n,n}\colon L_{s_a(j)}=z_n^{-d_a(j)}L_{j-1}\,\forall j\in O_a\}
$$
\end{proposition}

Here, $B^n$ is the set of all crossingless matchings of $2n$ points. Elements of $B^n$ can be thought of as diagrams consisting of $n$ disjoint, nested cups, as in Figure~\ref{fig:cups}. Equivalently, elements of $B^n$ are partitions of the set $\{1,2,\ldots,2n\}$ into pairs, such that there is no quadruple $i<j<k<l$ with $(i,k)$ and $(j,l)$ paired. For an element $a\in B^n$, we denote by $O_a$ the set of all $i$ appearing in a pair $(i,j)\in a$ with $i<j$; and if $(i,j)\in a$ is a pair with $i<j$, then we define $s_a(i):=j$ and $d_a(i):=(s_a(i)-i+1)/2$. Note that $d_a(i)$ is always an integer because $s_a(i)-i-1$ is twice the number of cups that are contained strictly inside the cup with endpoints $i$ and $s_a(i)$.

\newpage

In \cite{khovanov-2004-2}, Khovanov proved that the integer cohomology ring of $\mathfrak{B}_{n,n}$ is isomorphic to the center of the ring $H^n=\bigoplus_{a,b\in B^n}{}_b(H^n)_a$, defined in \cite{MR1928174}. To show this, Khovanov first proved that $\mathfrak{B}_{n,n}$ has the same integer cohomology ring as a topological space $\widetilde{S}\subset(\mathbb{P}^1)^{2n}=\mathbb{P}^1\times\ldots\times\mathbb{P}^1$ ($2n$ factors), defined by $\widetilde{S}:=\bigcup_{a\in B^n}S_a\subset(\mathbb{P}^1)^{2n}$, where
$$
S_a:=\{(l_1,\ldots,l_{2n})\in(\mathbb{P}^1)^{2n}\colon l_j=l_{s_a(j)}\,\forall j\in O_a\}.
$$
The goal of this paper is to show the following stronger statement, which was also conjectured by Khovanov (\cite[Conjecture~1]{khovanov-2004-2}):
\begin{theorem}\label{thm:main}
$\mathfrak{B}_{n,n}$ and $\widetilde{S}$ are homeomorphic.
\end{theorem}

\begin{figure}
\begin{picture}(30,15)(-12,-5)
\arc(-30,0){180}
\arc(-10,0){180}
\put(-33,4){\mbox{{\scriptsize $1$}}}
\put(-13,4){\mbox{{\scriptsize $2$}}}
\put(8,4){\mbox{{\scriptsize $3$}}}
\put(28,4){\mbox{{\scriptsize $4$}}}
\end{picture}\vspace*{0.5cm}

\caption{Crossingless matching $\{(1,4),(2,3)\}$.\label{fig:cups}}
\end{figure}
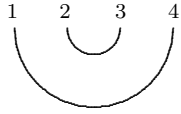

Our proof of Theorem~\ref{thm:main} is based on Proposition~\ref{prop:fung} and on the observation of Cautis and Kamnitzer \cite{cautis-2007} that $\mathfrak{B}_{n,n}$ can be embedded into a (smooth) complex projective variety $Y_{2n}$ diffeomorphic to $(\mathbb{P}^1)^{2n}$. Besides the diffeomorphism
$$
\phi_{2n}\colon Y_{2n}\longrightarrow(\mathbb{P}^1)^{2n}
$$
of Cautis and Kamnitzer, whose definition we review in Section~\ref{sec:diffeomorphism}, we will need an involutive diffeomorphism
$$
I_{2n}\colon(\mathbb{P}^1)^{2n}\longrightarrow(\mathbb{P}^1)^{2n}
$$
defined by $I_{2n}(l_1,\ldots,l_{2n}):=(l'_1,\ldots,l'_{2n})$ with
$$
l'_j:=\begin{cases}
l_j &\mbox{if $j$ is odd},\\
l_j^{\perp} &\mbox{if $j$ is even},
\end{cases}
$$
where $l_j^{\perp}\subset\C^2$ is the orthogonal complement (w.r.t. the standard hermitian product on $\C^2$) of the complex line $l_j\subset\C^2$ (or, equivalently, the antipode of the point $l_j\in\mathbb{P}^1\cong S^2$). In Section~\ref{sec:proof}, we prove the following result, which implies Theorem~\ref{thm:main}:

\begin{proposition}\label{prop:main} The diffeomorphism $I_{2n}\circ\phi_{2n}$ maps $K_a\subset Y_{2n}$ to $S_a\subset(\mathbb{P}^1)^{2n}$ for all $a\in B^n$, and hence $\mathfrak{B}_{n,n}$ to $\widetilde{S}$.
\end{proposition}

The author had the main idea for this article in Spring 2007 while he was preparing a talk for an informal seminar on link homology and coherent sheaves organized by Mikhail Khovanov at Columbia University. In a recent article~\cite{russell-2008}, Russell and Tymoczko studied an action of the symmetric group $S_{2n}$ on the cohomology ring of $\mathfrak{B}_{n,n}$. In this context, they also proved Theorem~\ref{thm:main}. Although our proof is similar to theirs, our work is completely independent.

{\em Acknowledgments.}
The author would like to thank Mikhail Khovanov for helpful conversations and for pointing him to the papers \cite{cautis-2007} and \cite{fung-2002}. The author was supported by fellowships of the Swiss National Science Foundation and of the Fondation Sciences Math\'ematiques de Paris.


\section{Diffeomorphism $\phi_m$}\label{sec:diffeomorphism}
In the following, $E$ is the complex vector space $E:=\C^N\oplus\C^N$ (for some $N>0$), and $z\colon E\rightarrow E$ is the nilpotent linear endomorphism given by $ze_j:=e_{j-1}$ and $zf_j:=f_{j-1}$ for all $j\in\{2,\ldots,N\}$, and $ze_1:=zf_1:=0$, where $\{e_1,\ldots,e_N\}$ is the standard basis for the first $\C^N$ summand in $E$, and $\{f_1,\ldots,f_N\}$ is the standard basis of the second $\C^N$ summand in $E$. For $n\leq N$, we denote by $E_n\subset E$ the subspace $E_n:=\C^n\oplus\C^n=\mbox{span}(e_1,\ldots,e_n)\oplus\mbox{span}(f_1,\ldots,f_n)$, or equivalently, $E_n=z^{-n}(0)=\mbox{ker}(z^n)=\mbox{im}(z^{N-n})$, and we denote by $\langle.,.\rangle_E$ the standard hermitian product on $E$, satisfying
$$\langle e_i,e_j\rangle_E :=\langle f_i,f_j\rangle_E:=\delta_{i,j}\quad,\quad \langle e_i,f_j\rangle_E :=0,$$
for all $i,j\in\{1,\ldots,N\}$, and by $\langle.,.\rangle$ the standard hermitian product on $\C^2$, satisfying
$$\langle e,e\rangle :=\langle f,f\rangle :=1\quad,\quad\langle e,f\rangle :=0,$$
where $\{e,f\}$ is the standard basis of $\C^2$.

\subsection{Stable subspaces}\label{subs:stable}
A subspace $W\subset E$ is called {\em stable} under $z$ if it satisfies $zW\subset W$. Note that this condition also implies $z^2W\subset zW$ and $W\subset z^{-1}W$, so if $W$ is stable under $z$, then so are its images and preimages under $z$. Moreover, if a stable subspace $W$ satisfies $W\subset \mbox{im}(z)$, then $z\colon z^{-1}W\rightarrow W$ is surjective and therefore
$$
\mbox{dim}((z^{-1}W)\cap W^{\perp})=\mbox{dim}(z^{-1}W/W)=\mbox{dim}(z^{-1}W)-\mbox{dim}(W)=\mbox{dim}(E_1)=2
$$
where we have used that $z^{-1}W\supset z^{-1}(0)=\mbox{ker}(z)=E_1$. Let $C\colon E\rightarrow\C^2$ be the linear map defined by $C(e_j):=e$ and $C(f_j):=f$ for all $j\in\{1,\ldots,N\}$. The following lemma is taken from \cite[Lemma~2.2]{cautis-2007}:

\begin{lemma}\label{lemma:isomorphism}
If $W\subset E$ is stable under $z$ and contained in $\mbox{{\normalfont im}}(z)$, then the restriction $C|_{(z^{-1}W)\cap W^{\perp}}\colon (z^{-1}W)\cap W^{\perp}\rightarrow \C^2$ is
an isomotric isomorphism.
\end{lemma}
 
For the convenience of the reader, we recall the proof given in \cite{cautis-2007}.

\begin{proof} Since $(z^{-1}W)\cap W^{\perp}$ is two-dimensional, it suffices to show that the restriction of $C$ to $(z^{-1}W)\cap W^{\perp}$ is an isometry. For this, let $v,w\in(z^{-1}W)\cap W^{\perp}$ with $v=v_1+\ldots+v_N$ and $w=w_1+\ldots+w_N$ where $v_j,w_j\in \text{span}(e_j,f_j)$. Then we have
$$\langle v,w\rangle_E=\sum_i\langle v_i,w_i\rangle_E=\sum_i\langle C(v_i),C(w_i)\rangle$$
and
$$\langle C(v),C(w)\rangle = \langle \sum_i C(v_i),\sum_j C(w_j)\rangle=\sum_{i,j}\langle C(v_i), C(w_j)\rangle.$$
To prove that the restriction of $C$ to $(zW)\cap W^{\perp}$ is an isometry, i.e. that $\langle v,w\rangle_E=\langle C(v),C(w)\rangle$, we must therefore show $\sum_{i\neq j}\langle C(v_i), C(w_j)\rangle=0$. We will actually prove a stronger statement, namely that $\sum_{i=j+k}\langle C(v_i), C(w_j)\rangle=0$ for each fixed $k\neq 0$. Assuming $k>0$ (the case $k<0$ being similar), we can write
$$
\sum_{i=j+k}\langle C(v_i), C(w_j)\rangle=
\sum_{i=j+k}\langle v_i, w_j\rangle_E= \langle v,z^kw\rangle_E,
$$
and since $v,w\in (z^{-1}W)\cap W^{\perp}$, we have $v\in W^{\perp}$ and $z^kw\in z^k(z^{-1}W)\subset z^{k-1}W\subset W$, whence $\langle v,z^kw\rangle_E=0$, as desired.
\end{proof}

\begin{lemma}\label{lemma:key}
Let $W\subset E$ be a stable subspace such that $\mbox{{\normalfont ker(z)}}\subset W\subset \mbox{{\normalfont im}}(z)$. Then $z$ maps $W^{\perp}\cap z^{-1}W$ isomorphically to $(zW)^{\perp}\cap W$, and the following diagram commutes:
\begin{equation*}
\xymatrix{
 {(z^{-1}W)\cap W^{\perp}}  \ar[dr]_C\ar[rr]^z &&
 {W\cap (zW)^{\perp}} \ar[dl]^C \\
 & {\C^2} &
}
\end{equation*}
\end{lemma}

\begin{proof}
It is apparent that $W\cap (zW)^{\perp}\cong W/(zW)$ is two-dimensional, and, by the previous lemma, $C$ restricts to an isomorphism on $(z^{-1}W)\cap W^{\perp}$, so we only need to prove that $z$ maps elements of $(z^{-1}W)\cap W^{\perp}$ to elements of $W\cap (zW)^{\perp}$, and that the above diagram commutes. Thus, let $v\in (z^{-1}W)\cap W^{\perp}$, and write $v$ as
$$v=v_1+\ldots+v_N$$
for $v_j\in\mbox{span}(e_j,f_j)$. Since $v\in W^{\perp}$ and $W\supset\mbox{ker}(z)=E_1=\mbox{span}(e_1,f_1)$, we have $v_1=0$, and since $C(zv_j)=C(v_j)$ for all $j\geq 2$, this implies $C(zv)=C(v)$. We clearly have $zv\in W$ (because $v\in z^{-1}W$), so the only thing that remains to be shown is that $zv\in (zW)^{\perp}$. For this, consider any $w\in W$ and write $w$ as $w=w_1+\ldots+w_N$ for $w_j\in\mbox{span}(e_j,f_j)$. Since $\langle zv_j,zw_j\rangle_E=\langle v_j,w_j\rangle_E$ for all $j\geq 2$, and since $v_1=0$ and $v\in W^{\perp}$, we see that $\langle zv,zw\rangle_E=\langle v,w\rangle_E=0$, and thus $zv\in (zW)^{\perp}$.
\end{proof}

\subsection{$Y_m$ and $\phi_m$}\label{subs:yn}
For $m\leq N$, Cautis and Kamnitzer~\cite[Section~2]{cautis-2007} define a complex projective variety $Y_m$,
$$
Y_m:=\{(L_1,\ldots,L_m)\in F_m\,\colon\,\mbox{dim}(L_j)=j\,\text{ and }\,zL_j\subset L_j\,\forall j\},
$$
where $F_m$ is the set of all partial flags $0\varsubsetneq L_1\varsubsetneq L_2\varsubsetneq\ldots\varsubsetneq L_m\subset E$. Note that the conditions $zL_j\subset L_j$ and $zL_{j-1}\subset L_{j-1}$ imply that $z$ descends to an endomorphism of $L_j/L_{j-1}$, and since $L_j/L_{j-1}$ is one-dimensional and $z$ nilpotent, this endomorphism must be the zero-map, so the spaces $L_j$ in $(L_1,\ldots,L_m)\in Y_m$ actually satisfy the seemingly stronger condition $zL_j\subset L_{j-1}$. In particular, $L_m\subset z^{-1}L_{m-1}\subset z^{-2}L_{m-2}\subset\ldots\subset z^{-m}(0)=\mbox{ker}(z^m)=E_m$, so as far as the definition of $Y_m$ is concerned, we could restrict ourselves to the space $E_m=\C^m\oplus\C^m$ instead of working with the bigger space $E=\C^N\oplus\C^N$. In particular, $Y_m$ is independent of the choice of $N$ (as long as $N\geq m$).

Note also that the assignment $(L_1,\ldots,L_{m-1},L_m)\mapsto (L_1,\ldots,L_{m-1})$ defines a $\mathbb{P}^1$-bundle $Y_m\rightarrow Y_{m-1}$. Indeed, a point in the fiber above $(L_1,\ldots,L_{m-1})\in Y_{m-1}$ is obtained from $(L_1,\ldots,L_{m-1})$ by choosing an $L_m$ such that $L_{m-1}\subset L_m\subset z^{-1}L_{m-1}$, and since $z^{-1}L_{m-1}/L_{m-1}$ is two-dimensional, we have a $\mathbb{P}^1$ worth of choices. Denoting by $L_{j-1}^{\perp}$ the orthogonal complement of $L_{j-1}$ w.r.t. $\langle.,.\rangle_E$, we can identify $z^{-1}L_{m-1}/L_{m-1}$ with $(z^{-1}L_{m-1})\cap L_{m-1}^{\perp}$, and by Lemma~\ref{lemma:isomorphism}, the map $C\colon E\rightarrow\C^2$ identifies $(z^{-1}L_{m-1})\cap L_{m-1}^{\perp}$ with $\C^2$. Therefore, the $\mathbb{P}^1$-bundle $Y_m\rightarrow Y_{m-1}$ is topologically trivial (i.e., topologically, $Y_m\cong\mathbb{P}^1\times Y_{m-1}$), and Cautis and Kamnitzer use this to define a diffeomorphism
$$
\phi_m\colon Y_m\longrightarrow (\mathbb{P}^1)^m
$$
by $\phi_m(L_1,\ldots,L_m):=(C(L_1),C(L_2\cap L_1^{\perp}),C(L_3\cap L_2^{\perp}),\ldots,C(L_m\cap L_{m-1}^{\perp}))$.

\subsection{Subvarieties $X_{m,i}\subset Y_m$}

For each $i\in\{1,\ldots,m-1\}$, Cautis and Kamnitzer \cite[Section~2]{cautis-2007} define a subvariety $X_{m,i}\subset Y_m$,
$$
X_{m,i}:=\{(L_1,\ldots,L_m)\in Y_m\,\colon\,L_{i+1}=z^{-1}(L_{i-1})\},
$$
together with a surjection
$$
q_{m,i}\colon X_{m,i}\longrightarrow Y_{m-2},
$$
given by $q_{m,i}(L_1,\ldots,L_m):=(L_1,\ldots,L_{i-1},zL_{i+2},\ldots,zL_m)\in Y_{m-2}$. The following (easy) Lemma was shown in \cite[Theorem~2.1]{cautis-2007}.

\begin{lemma}
The map $\phi_m\colon Y_m\rightarrow(\mathbb{P}^1)^m$ takes $X_{i,m}$ diffeomorphically to
$$A_{m,i}:=\{(l_1,\ldots,l_m)\in (\mathbb{P}^1)^m\,\colon\,l_{i+1}=l_i^{\perp}\},$$
where $l_i^{\perp}$ denotes the orthogonal complement of the line $l_i\subset\mathbb{C}^2$ w.r.t. $\langle.,.\rangle$.
\end{lemma}

Let $f_{m,i}\colon (\mathbb{P}^1)^m\rightarrow (\mathbb{P}^1)^{m-2}$ be the forgetful map sending $(l_1,\ldots,l_m)\in (\mathbb{P}^1)^m$ to $(l_1,\ldots,l_{i-1},l_{i+2},\ldots,l_m)\in (\mathbb{P}^1)^{m-2}$, and let $$g_{m,i}:A_{m,i}\longrightarrow (\mathbb{P}^1)^{m-2}$$ be the restriction of $f_{m,i}$ to $A_{m,i}$.

\begin{lemma}\label{lemma:commute}
Let $\psi_{m,i}\colon X_{m,i}\rightarrow A_{m,i}$ be the restriction of $\phi_{m}$ to $X_{m,i}\subset Y_{m}$. Then the following diagram commutes:
$$
\xymatrix{
X_{m,i}\ar[r]^{q_{m,i}}\ar[d]^{\psi_{m,i}} &Y_{m-2}\ar[d]^{\phi_{m-2}}\\
A_{m,i}\ar[r]^{g_{m,i}} &(\mathbb{P}^1)^{m-2}
}
$$
\end{lemma}
\begin{proof}
It is straightforward to check that $g_{m,i}\circ\psi_m$ maps $(L_1,\ldots,L_m)\in X_{m,i}$ to the tuple $(l'_1,\ldots,l'_{m-2})\in(\mathbb{P}^1)^{m-2}$, where
$$
l'_j=
\begin{cases}
C(L_j\cap L_{j-1}^{\perp}) &\mbox{if }\, j<i,\\
C(L_{j+2}\cap L_{j+1}^{\perp}) &\mbox{if }\, j\geq i,
\end{cases}
$$
and $\phi_{m-2}\circ q_{m,i}$ maps $(L_1,\ldots,L_m)\in X_{m,i}$ to the tuple $(l''_1,\ldots,l''_{m-2})\in(\mathbb{P}^1)^{m-2}$, where
$$
l''_j=
\begin{cases}
C(L_j\cap L_{j-1}^{\perp}) &\mbox{if }\, j<i,\\
C(zL_{j+2}\cap (zL_{j+1})^{\perp}) &\mbox{if }\, j\geq i.
\end{cases}
$$
To prove $g_{m,i}\circ \psi_m=\phi_{m-2}\circ q_{m,i}$, we must therefore show that
$$C(L_{j+2}\cap L_{j+1}^{\perp})=C(zL_{j+2}\cap (zL_{j+1})^{\perp})$$ holds for all $j\geq i$.
But if $j\geq i$, then $L_{j+1}\supset L_{i+1}=z^{-1}L_{i-1}\supset z^{-1}(0)=\mbox{ker}(z)$, and (by increasing $N$ if necessary) we can also assume that $L_{j+1}\subset\mbox{im}(z)$. Thus, Lemma~\ref{lemma:key} applied to $W:=L_{j+1}$ tells us that $z$ maps $(z^{-1}W)\cap W^{\perp}$ to $W\cap (zW)^{\perp}$, and that $C(v)=C(zv)$ for all $v\in (z^{-1}W)\cap W^{\perp}$. Now the equality $C(L_{j+2}\cap L_{j+1}^{\perp})=C(zL_{j+2}\cap (zL_{j+1})^{\perp})$ follows because $z$ maps $L_{j+2}\cap L_{j+1}^{\perp}\subset  (z^{-1}W)\cap W^{\perp}$ to $zL_{j+2}\cap (zL_{j+1})^{\perp}\subset W\cap (zW)^{\perp}$.
\end{proof}


\section{Proof of Proposition~\ref{prop:main}}\label{sec:proof}
In this section, we use the same notations as before, except that we now assume $m=2n$ (and hence $N\geq 2n$). Then the Springer variety $\mathfrak{B}_{n,n}$ is naturally contained in $Y_{2n}$ as
$$
\mathfrak{B}_{n,n}:=\{(L_1,\ldots,L_{2n})\in Y_{2n}\,\colon\,L_{2n}=E_n\},
$$
where $E_n:=\mbox{span}(e_1,\ldots,e_n)\oplus\mbox{span}(f_1,\ldots,f_n)$, and Proposition~\ref{prop:fung} tells us that the irreducible component $K_a\subset\mathfrak{B}_{n,n}\subset Y_{2n}$ associated to the crossingless matching $a\in B^n$ is equal to the set of all $(L_1,\ldots,L_{2n})\in Y_{2n}$ satisfying
$$
L_{s_a(j)}=z_n^{-d_a(j)}L_{j-1}
$$
for all $j\in O_a$, where $z_n\colon E_n\rightarrow E_n$ is the restriction of $z$ to $E_n$. A priori, $z_n^{-d_a(j)}L_{j-1}$ could a priori be a proper subspace of $z^{-d_a(j)}L_{j-1}$ (because $z^{-d_a(j)}L_{j-1}$ might not be contained in $E_n$), but it turns out that
$z_n^{-d_a(j)}L_{j-1}$ is equal to $z^{-d_a(j)}L_{j-1}$ whenever $(L_1,\ldots,L_{2n})\in K_a$. In fact, we have:

\begin{lemma} 
$K_a=\{(L_1,\ldots,L_{2n})\in Y_{2n}\,\colon\, L_{s_a(j)}=z^{-d_a(j)}L_{j-1}\,\forall j\in O_a\}$.
\end{lemma}

\begin{proof} Suppose $(L_1,\ldots,L_{2n})$ is contained in $K_a$. Then the condition $z_n^{-d_a(j)}L_{j-1}=L_{s_a(j)}$, combined with $\mbox{dim}(L_{j-1})=j-1$, $\mbox{dim}(L_{s_a(j)})=s_a(j)$, and $\mbox{dim}(\mbox{ker}(z))=2$, implies
\begin{eqnarray*}
\mbox{dim}(z^{-d_a(j)}L_{j-1})&=&2d_a(j)+\mbox{dim}(L_{j-1})=2d_a(j)+j-1=s_a(j)\\
&=&\mbox{dim}(L_{s_a(j)})=\mbox{dim}(z_n^{-d_a(j)}L_{j-1}),
\end{eqnarray*}
and thus $z^{-d_a(j)}L_{j-1}=z_n^{-d_a(j)}L_{j-1}$. Conversely, suppose $(L_1,\ldots,L_{2n})\in Y_{2n}$ satisfies $z^{-d_a(j)}L_{j-1}=L_{s_a(j)}$ for all $j\in O_a$. Then we must show that $L_{2n}=E_n$.
To prove this, let us call a pair $(k,l)\in a$ {\em outermost} if there is no pair $(k',l')\in a$ such that $k'<k<l<l'$. Then the outermost pairs in $a$ form a sequence $(k_1,l_1),(k_2,l_2),\ldots,(k_r,l_r)\in a$ such that $k_1=1$, $l_r=2n$, and $k_{s+1}=l_s+1$ for all $s<r$, and $d_a(k_1)+\ldots+d_a(k_r)=n$. Using $z^{-d_a(j)}L_{j-1}=L_{s_a(j)}$ successively for $j\in\{k_r,k_{r-1},\ldots,k_1\}\subset O_a$, we obtain
$$
L_{2n}=z^{-d_a(k_r)}L_{l_{r-1}}=z^{-d_a(k_r)}z^{-d_a(k_{r-1})}L_{l_{r-2}}=\ldots=z^{-n}(0)=E_n,
$$
as desired.
\end{proof}

From now on, $a\in B^n$ is a fixed crossingless matching of $2n$ points, and $i$ is an index such that $s_a(i)=i+1$, i.e., such that $(i,i+1)$ is a pair in $a$. We denote by $a'\in B^{n-1}$ the crossingless matching obtained from $a$ by removing the pair $(i,i+1)$ (and renumbering indices $j\geq i+2$ such that $j\in\{i+2,\ldots,2n\}$ becomes $j-2\in\{i,\ldots,2n-2\}$), and by $q$ the map $q_{2n,i}\colon X_{2n,i}\rightarrow Y_{2n-2}$, defined as in the previous section.

\begin{lemma}\label{lemma:q}
$K_a=q^{-1}(K_{a'})$.
\end{lemma}

\begin{proof} Since $s_a(i)=i+1$ and $d_a(i)=(s_a(i)-i+1)/2=1$, the equality $L_{i+1}=z^{-1}L_{i-1}$ holds for each $(L_1,\ldots,L_{2n})\in K_a$, and thus $K_a\subset Y_{2n}$ is contained in $X_{2n,i}$. It remains to show that an element $(L_1,\ldots,L_{2n})\in X_{2n,i}$ satisfies the conditions $L_{s_a{j}}=z^{-d_a(j)}L_{j-1}$ for all $j\in O_a\setminus\{i\}$ if and only if the element $(L'_1,\ldots,L'_{2n-2}):=q(L_1,\ldots,L_{2n})=(L_1,\ldots,L_{i-1},zL_{i+2},\ldots,zL_{2n})\in Y_{2n-2}$ satisfies the conditions $L'_{s_{a'}(j)}=z^{-d_{a'}(j)}L'_{j-1}$ for all $j\in O_{a'}$. We divide the proof into three cases.

{\em Case 1.} If $j<s_a(j)<i$, then the equivalence
$$L_{s_a(j)}=z^{-d_a(j)}L_{j-1}\Longleftrightarrow L'_{s_{a'}(j)}=z^{-d_{a'}(j)}L'_{j-1}$$
is obvious because the quantities appearing on either side of $\Longleftrightarrow$ are identical.

{\em Case 2.} If $j<i<i+1<s_a(j)$, then $L'_{j-1}=L_{j-1}$, $L'_{s_{a'}(j)}=zL_{s_a(j)}$, and $d_{a'}(j)=d_a(j)-1$, so we must show:
$$
L_{s_a(j)}=z^{-d_a(j)}L_{j-1}\Longleftrightarrow zL_{s_a(j)}=z^{-d_{a}(j)+1}L_{j-1}
$$
But this follows simply by applying $z$ (resp., $z^{-1}$) to the equalities on either side of $\Longleftrightarrow$, and observing that $z^{-1}(zL_{s_a(j)})=L_{s_a(j)}$ (because $L_{s_a(j)}\supset L_{i+1}=z^{-1}L_{i-1}\supset z^{-1}(0)=\mbox{ker}(z)$), and that $z(z^{-d_a(j)}L_{j-1})=z^{-d_{a}(j)+1}L_{j-1}$ (because, by increasing $N$ if necessary, we may assume $z^{-d_{a}(j)+1}L_{j-1}\subset\mbox{im}(z)$).

{\em Case 3.} If $i+1<j<s_a(j)$, then $L'_{j-3}=zL_{j-1}$, $L_{s_{a'}(j-2)}=zL_{s_a(j)}$, and $d_{a'}(j-2)=d_a(j)$, so we must show:
$$
L_{s_a(j)}=z^{-d_a(j)}L_{j-1}\Longleftrightarrow\
zL_{s_a(j)}=z^{-d_a(j)}zL_{j-1}
$$
As in Case~2, this follows by applying $z$ (resp., $z^{-1}$) to the equalities on either side of $\Longleftrightarrow$.\end{proof}

Note that (since $s_a(j)-j$ is odd for all $j\in O_a$) the involutive diffeomorphism $I_{2n}\colon(\mathbb{P}^1)^{2n}\rightarrow(\mathbb{P}^1)^{2n}$ defined in the introduction exchanges the subset $S_a\subset(\mathbb{P}^1)^{2n}$ with the subset
$$
T_a:=\{(l_1,\ldots,l_{2n})\in(\mathbb{P}^1)^{2n}\,\colon\,l_{s_a(j)}=l_j^{\perp}\,\forall j\in O_a\}\,\subset\,(\mathbb{P}^1)^{2n}
$$
To prove Proposition~\ref{prop:main}, we must therefore show that $\phi_{2n}$ maps $K_a$ to $T_a$ for all $a\in B^n$. We will need the following lemma, in which $a$, $i$ and $a'$ are as in the previous lemma, and $g$ denotes the map $g_{2n,i}\colon A_{2n,i}\rightarrow (\mathbb{P}^1)^{2n-2}$, defined as in the previous section.

\begin{lemma}\label{lemma:g}
$T_a=g^{-1}(T_{a'})$.
\end{lemma}

\begin{proof} This follows directly from the definitions of $g$, $A_{2n,i}$, $T_a$ and $T_{a'}$.
\end{proof}

We are now ready to prove Proposition~\ref{prop:main}.

\begin{proof}[Proof of Proposition~\ref{prop:main}]
Induction on $n$. The case $n=1$ is trivial because the only crossingless matching of $2$ points is $a_1:=\{(1,2)\}$, and $\phi_2\colon Y_2\rightarrow \mathbb{P}^1\times\mathbb{P}^1$ maps $\mathfrak{B}_{1,1}=K_{a_1}=X_{2,1}\subset Y_2$ diffeomorphically to $T_{a_1}=A_{2,1}\subset\mathbb{P}^1\times\mathbb{P}^1$.

Thus, let $n>1$, and suppose we have already proven the proposition for $n-1$. Let $a\in B^n$. Then there is an $i\in\{1,\ldots,2n-1\}$ such that $s_a(i)=i+1$, i.e., such that $(i,i+1)\in a$. As above, we denote by $a'\in B^{n-1}$ the crossingless matching obtained from $a$ by removing the pair $(i,i+1)$ (and renumbering all $j\geq i+2$), and by $q$ (resp., $g$) the map $q_{2n,i}$ (resp., $g_{2n,i}$). By induction, we know that $\phi_{2n-2}$ maps $K_{a'}$ to $T_{a'}$, so Lemma~\ref{lemma:commute} gives us the following commutative diagram:
$$
\xymatrix@M=0.2cm{
q^{-1}(K_{a'})\ar[d]^{\psi_{2n,i}
}
\ar@{^{(}->}[r] &X_{2n,i}
\ar[r]^{q}\ar[d]^{\psi_{2n,i}} &Y_{2n-2}\ar[d]^{\phi_{2n-2}}
&K_{a'}\ar@{_{(}->}[l]\ar[d]^{\phi_{2n-2}
}
\\
g^{-1}(T_{a'})\ar@{^{(}->}[r] &A_{2n,i}
\ar[r]^{g} &(\mathbb{P}^1)^{2n-2} &T_{a'}\ar@{_{(}->}[l]
}
$$
Hence we get $\psi_{2n,i}(q^{-1}(K_{a'}))=g^{-1}(T_{a'})$, and by Lemmas~\ref{lemma:q} and \ref{lemma:g}, this implies $$\psi_{2n,i}(K_a)=T_a,$$ thus completing the inductive step.
\end{proof}

\bibliography{springer}
\end{document}